\documentclass[english, 12pt]{amsart}
\usepackage{nicefrac}
\usepackage[inline]{enumitem}
\usepackage{enumitem}
\usepackage{euscript}
\usepackage{marvosym}
\usepackage{mathrsfs}
\usepackage{xkeyval}
\usepackage{moreverb}
\usepackage{pgf}
\usepackage{epic}
\usepackage{accents}
\usepackage{todonotes}
\usepackage{amssymb}
\usepackage{mathrsfs}

\definecolor{myRed}{rgb}{0.9,0.,.2}
\definecolor{myBlue}{rgb}{0.,0.,.6}
\definecolor{myGreen}{rgb}{0.1,0.7,0.1}
\definecolor{myViolet}{rgb}{102,0,153}
\usepackage[colorlinks=true, urlcolor=myBlue,linkcolor=myBlue, citecolor=myGreen, pagebackref=true]{hyperref}

\usepackage{fouriernc}
\usepackage[scaled=0.775]{helvet}
\usepackage{courier}

\usepackage[marginparwidth=2cm]{geometry}
\usepackage{array,float}

\usepackage{tikz}
\usetikzlibrary{shapes,arrows}
\usetikzlibrary{fit,positioning}
\usetikzlibrary{patterns,decorations.pathreplacing}
\usepackage{xkeyval}
\usepackage{moreverb}
\usepackage{pgf}
\usepackage{epic}
\usepgfmodule{shapes,plot,decorations}

\usepackage{booktabs} 
\usepackage[normalem]{ulem}

\newcommand{\Z}{\mathbb{Z}}
\newcommand{\Q}{\mathbb{Q}}
\newcommand{\R}{\mathbb{R}}

\newcommand{\SSS}{\mathbb{S}}

\newcommand{\HH}{\mathbb{H}}

\usepackage[warn]{mathtext}

\newcommand{\OOO}{\mathcal{O}}

\renewcommand{\H}{\mathcal{H}}

\newcommand{\Isom}{\mathrm{Isom}}

\newcommand{\tG}{\widetilde{\Gamma}}

\newcommand{\OO}{\mathrm{O}}

\newcommand{\Cyc}{\mathrm{Cyc}}

\theoremstyle{plain}
\newtheorem{theorem}{\indent Theorem}[section]
\newtheorem{question}[theorem]{\indent Question}
\newtheorem{proposition}[theorem]{\indent Proposition}

\newtheorem{lemma}[theorem]{\indent Lemma}

\newtheorem{theoremA}{\indent Theorem}
\newtheorem{theoremB}{\indent Theorem}
\newtheorem{theoremC}{\indent Theorem}

\renewenvironment{proof}
	{\par\indent{\bf Proof.}} 
	{\hfill$\scriptstyle\blacksquare$}

\makeatletter
  \def\section{\@startsection{section}{2}%
    {\z@}{.5\linespacing\@plus.7\linespacing}{.5em}%
    {\normalfont\bfseries\centering}}
\def\@secnumfont{\bfseries}
\makeatother

\theoremstyle{definition}

\newtheorem{definition}[theorem]{\indent Definition}

\theoremstyle{remark}

\title[Kleinian sphere packings, reflection groups, and arithmeticity]{Kleinian sphere packings, reflection groups, \\ and arithmeticity}

\author{Nikolay Bogachev}
\thanks{Bogachev was partially supported by the Russian Science Foundation, grant no. 22-41-02028.}
\address{The Kharkevich Institute for Information Transmission Problems, Moscow, Russia}
\address{Moscow Institute of Physics and Technology, Dolgoprudny, Russia}
\email{nvbogach@mail.ru}

\author{Alexander Kolpakov}
\thanks{Kolpakov was partially supported by the Swiss National Science Foundation, project no. PP00P2-202667 }
\address{Institut de Math\'ematiques, Universit\'e de Neuch\^atel, CH--2000 Neuch\^atel, Switzerland}
\email{kolpakov.alexander@gmail.com}

\author{Alex Kontorovich}
\thanks{Kontorovich was partially supported by NSF grant DMS-1802119, BSF grant 2020119, and the
2020-2021 Distinguished Visiting Professorship at the National Museum of Mathematics.}
\address{Department of Mathematics, Rutgers University, 110 Frelinghuysen Rd, Piscataway, NJ 08854, USA}
\email{alex.kontorovich@rutgers.edu}

\begin{document}

\begin{abstract}
    In this paper we study crystallographic sphere packings and Kleinian sphere packings, introduced first by Kontorovich and Nakamura in 2017 and then studied further by Kapovich and Kontorovich in 2021. In particular, we solve the problem of existence of crystallographic sphere packings in certain 
    higher dimensions posed by Kontorovich and Nakamura. In addition, we present a  geometric doubling procedure allowing to obtain sphere packings from some Coxeter polyhedra without isolated roots, and study ``properly integral'' packings (that is, ones which are integral but not superintegral). Our techniques rely extensively on computations with Lorentzian quadratic forms, their orthogonal groups, and associated higher--dimensional hyperbolic polyhedra.
\end{abstract}

\large
\maketitle

\section{Introduction}

\subsection{Sphere packings}
For $n \geqslant 2$, let $\HH^{n+1}$ be the $(n+1)$--di\-men\-sio\-nal hyperbolic space. Then the set of all ideal points of hyperbolic space, or its ideal boundary, is $\partial \HH^{n+1} \approx \SSS^{n}$ and can be identified topologically with $\overline{\R^n}:= \R^n \cup \{\infty\}$. One can consider different coordinate systems in $\SSS^n$, and thus some of the definitions below depend on this choice of coordinates.

\begin{definition}
By a \textit{sphere packing} (or just \textit{packing}) $\mathscr{P}$ of $\SSS^n$
we mean an infinite collection $\{B_\alpha\}$ of 
round balls  in $\SSS^n$ so that: 
\begin{itemize}
    \item[\bf (1)] The interiors of the balls are disjoint, and 
    \item[\bf (2)] The union of the balls is dense in $\SSS^n$.
\end{itemize}
\end{definition}

\begin{figure}
    \centering
    \includegraphics[width=3in]{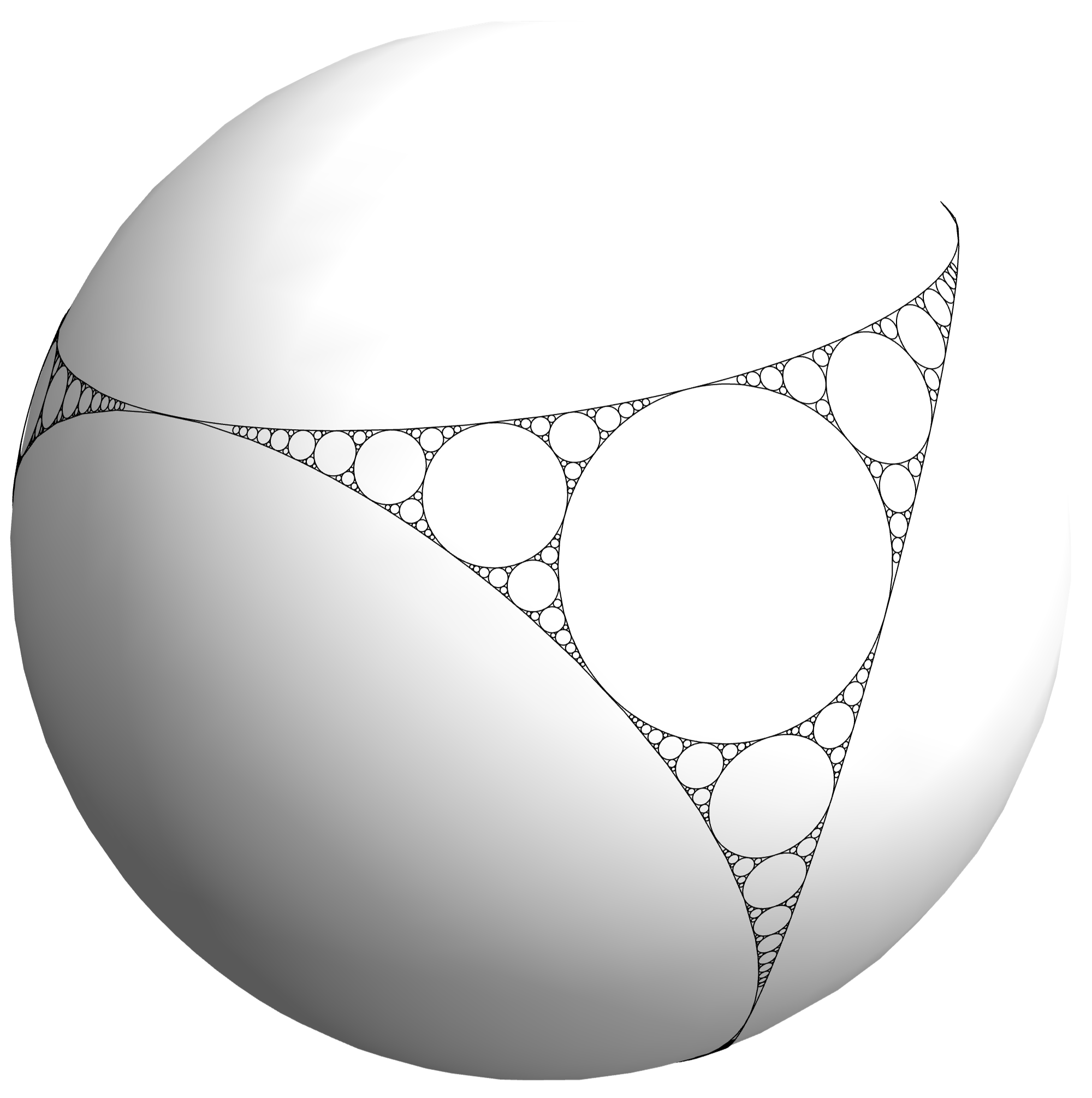}
    \caption{The Apollonian sphere packing. (Image by Iv\'{a}n Rasskin.) }
    \label{fig:Raskin}
\end{figure}
 
See Figure \ref{fig:Raskin} for an example.
We will often conflate balls in a packing with the round spheres bounding them.
Recall that, given a packing $\mathscr{P} = \{B_\alpha\}$ of $\SSS^n \approx \partial \HH^{n+1}$, every sphere $S_\alpha = \partial B_\alpha$ is also the boundary of a hyperplane $H_\alpha \cong \HH^n$ of $\HH^{n+1}$, i.e.,
$S_\alpha = \partial_{\infty} H_\alpha$. Any such hyperplane $H_\alpha$ is associated to a reflection $R_\alpha \in \Isom(\HH^{n+1})$, i.e., $H_\alpha$ is the mirror of the reflection $R_\alpha$. Let $\Gamma_R$ be the group generated by the reflections in these hyperplanes, $\Gamma_R:=\langle R_\alpha \mid B_\alpha \in \mathscr{P} \rangle.$

\begin{definition}
The \textit{superpacking} of $\mathscr{P}$ is its orbit under the $\Gamma_R$--action: 
$$
\widetilde{\mathscr{P}} = \Gamma_R \cdot \mathscr{P}.
$$
\end{definition}

\begin{definition}
 Choosing a point at infinity for the identification $\SSS^n=\R^n\cup\{\infty\}$ allows one to put a corresponding Euclidean metric on the sphere.
The {\it bend} of a ball $B_\alpha\in \mathscr{P}$ is the inverse of its signed radius in these coordinates. A radius is negative if the point at infinity is inside the ball.
\end{definition}

\begin{definition}
A packing is {\it integral} if there is a choice of such an identification $\SSS^n=\R^n\cup\{\infty\}$ that all balls $B_\alpha$ have integer bends. A packing is \textit{superintegral} if its superpacking has all integer bends.
We say that a packing is \textit{properly integral} if it is integral but not superintegral.
\end{definition}

\begin{definition}
A packing $\mathscr{P}$ is \textit{Kleinian} if the set of accumulation points of the spheres bounding its balls is the limit set of some geometrically finite discrete group
$\Gamma_S < \Isom(\HH^{n+1})$. We call such a $\Gamma_S$ {\it a symmetry group} of $\mathscr{P}$.
\end{definition}

\begin{definition}
Let $\mathscr{P} = \{B_\alpha\}$ be a Kleinian sphere packing with a symmetry group $\Gamma_S < \Isom(\HH^{n+1})$. Then \textit{a supergroup} of $\mathscr{P}$ is the group $\widetilde{\Gamma} = \Gamma_R \rtimes \Gamma_S$ generated by a symmetry group $\Gamma_S$ and all reflections $R_\alpha$.
\end{definition}
 
It is known that $\widetilde{\Gamma}$ is a lattice, that is, acts on $\HH^{n+1}$ with finite covolume; see the Structure Theorem (Theorem~\ref{theorem:structure-kleinian} below). Note that a symmetry group $\Gamma_S$ and supergroup $\widetilde \Gamma$ are not unique (indeed, compare Figure 2B and Figure 5A in \cite{KN}).

\begin{definition}
A packing is \textit{crystallographic} if it has a {\it reflective} symmetry group, that is, one generated by reflections in hyperplanes.
\end{definition}

\subsection{Outline and main theorems} The main focus of this paper is on determining the dimensions in which superintegral crystallographic sphere packings exist. It is worth mentioning that Kleinian sphere packings, even those that are superintegral, exist in all dimensions \cite{K^2}. 

The existence of a superintegral crystallographic packing in $\SSS^n$, $n\geqslant 2$, implies that there exists a reflective non--uniform arithmetic lattice of simplest type defined over $\mathbb{Q}$ that acts on $\HH^{n+1}$. (That it is reflective and arithmetic over $\mathbb Q$ follows directly from the Structure Theorem and Arithmeticity Theorem in \cite{KN}. That it is non--uniform is proved in \cite{K^2}.) A result of Esselmann \cite{E} implies that no such lattice is possible for $n\geqslant 21$. Previously, superintegral crystallographic sphere packings were known in dimensions $n\leqslant 13$ and $n=17$ by applying the techniques of \cite{KN}. We also remark that the existence of such a packing in dimension $n=18$ is claimed in \cite{KN}, although a further argument is needed  (given below in Lemma~\ref{lemma:doubling}) to substantiate this claim. 

The difficulty in producing a crystallographic sphere packing stems from the Structure Theorem of \cite{KN}. There it is shown that a finite covolume reflection group $\widetilde \Gamma$ gives rise to a sphere packing of $\SSS^n$ if and only if its Coxeter--Vinberg diagram (see \S\ref{sec:Coxeter}) has at least one ``isolated root''; by this we mean, that the corresponding facet  is orthogonal to all its neighbors. In terms of the Coxeter--Vinberg diagram, this means that the corresponding (isolated) vertex is connected to its neighbors only by dotted (common perpendiculars) or bold edges (parallel, i.e. tangent at the ideal boundary $\partial \mathbb{H}^{n+1}$).

Our main result about dimensions of superintegral crystallographic sphere packings is as follows. 

\begin{theoremA}\label{theorem:dimension-bound}
There are no superintegral crystallographic sphere packings in $\SSS^{n}$ for $n\geqslant 19$. 
\end{theoremA}

The theorem above is implied by the results of Esselmann \cite{E}, together with the following fact. 

\begin{proposition}\label{prop:lower-dimension}
If a superintegral crystallographic sphere packing exists in $\SSS^{n}$, then there exists a $\Q$--arithmetic reflective lattice acting on $\HH^n$.
\end{proposition}

That is, not only is there a reflective lattice acting on $\HH^{n+1}$, namely, the supergroup, but there is moreover a reflective lattice in one dimension lower.

As for the existence of crystallographic sphere packings, we can summarise our findings together with the previous results of \cite{KN}, in order to formulate the following statement.

\begin{theoremB}\label{theorem:crystallographic}
Superintegral crystallographic sphere packings in $\SSS^n$ exist only for $n\leqslant 18$. Examples are known for all $2 \leqslant n \leqslant 14$, and $n = 17, 18$.
\end{theoremB}

What is new here is the addition of dimensions $n=14$ and $n=18$. We construct these additional examples in higher dimensions by studying the lower--dimensional faces of the $21$--dimensional hyperbolic polyhedron discovered by Borcherds \cite{Bor}.  
 
The proof of Theorem~\ref{theorem:crystallographic} relies on the following fact that allows to combine Coxeter polyhedra with many ``even'' angles into one that has a facet orthogonal to all its neighbours. 

\begin{lemma}\label{lemma:doubling}
Let $P$ be a finite volume Coxeter polyhedron in $\HH^{n+1}$ and $P_0$ be a facet that meets other facets of $P$ at angles  $\pi/2$ or $\pi/4$. Let $P_1, \ldots P_k$ be all the facets of $P$ that meet $P_0$ at $\pi/4$.
Assume that each $P_j$ meets its neighbors at angles $\pi/2^m$, with $m \geqslant 1$, an integer.
Then the reflection group of $P$ is commensurable to the supergroup of a crystallographic packing.
\end{lemma}

Dimensions $n = 15$ and $16$ are still missing from our consideration: the gap seems non--trivial to fill, as at present, there are rather few candidate arithmetic lattices in high dimensions to explore for packings. For example, we studied another polyhedron discovered by Borcherds \cite{Borcherds2000} in $\HH^{17}$ (with $960$ facets), only to discover that it has no isolated roots, and no Coxeter facets either. Thus the present methods of obtaining sphere packings from polyhedra (Lemma~\ref{lemma:doubling} and the work in \cite{KN}) cannot be applied to the polyhedron itself.  We would like to thank Daniel Allcock for his invaluable help in our efforts to transfer this example into a computer--readable format. 

\begin{question}\label{q:exist}
Do there exist crystallographic (not necessarily integral or superintegral) sphere packings in $\SSS^n$ for $n = 15$ and $16$?
\end{question}

We should stress the fact that answering the above question in the negative even in the more narrow superintegral setting will require, at minimum, a complete classification of reflective quadratic Lorentzian forms over $\mathbb{Q}$. There is ongoing work in progress of Kirschmer and Schalau \cite{KS} that aims to complete this classification. They have kindly shared their data, and none of the lattices in $\HH^{16}$ and $\HH^{17}$ produce packings via our methods (such as Lemma \ref{lemma:doubling}, though we cannot rule out the possibility of some hypothetical reflective subgroup of such which supports a packing). As a result, it seems that the answer to Question \ref{q:exist} may be negative, in which case Theorem \ref{theorem:crystallographic} is sharp.

Another interesting direction to pursue is the possible connection between arithmetic reflection groups and sphere packings. Namely, Kontorovich and Nakamura \cite{KN} asked the following question:

\begin{question}\label{question}
Is it true that every non-cocompact arithmetic reflection group in $\H^{n+1}$ is commensurable to the supergroup of a superintegral crystallographic packing of $S^n$?
\end{question}
 
The answer turns out to be affirmative in low dimensions: here we confirm it for $n = 2, 3, 4$ by using the classification results of Scharlau, Walhorn, and Turkalj. 

\begin{theoremC}\label{theorem:arithmetic-super}
Any  $\Q$--arithmetic non--uniform   reflective lattice $\Gamma < \mathrm{Isom}\,\mathbb{H}^{n+1}$, $n = 2,3,4$,  is commensurable to the supergroup of a superintegral crystallographic packing in $\SSS^{n}$. 
\end{theoremC}
For $n=2$, Kontorovich and Nakamura \cite{KN} confirmed the analogue of Theorem \ref{theorem:arithmetic-super} for the finite list of reflective extended Bianchi groups. However, the Borcherds polyhedron in dimension $21$ provides an example of a negative answer to Question \ref{question}. 

\subsection{Structure of the paper} Section~\ref{sec:Coxeter} contains all necessary preliminaries on integral, superintegral and Kleinian sphere packings. There we also recall Vinberg's arithmeticity criterion and algorithm. Sections~\ref{sec:A}, \ref{sec:B} and \ref{sec:C} contain the proofs of our main theorems. Finally, in Section~\ref{sec:remark-refl} we provide an example of a properly integral packing and study its properties.

\subsection*{Acknowledgements}

The authors would like to thank Daniel Allcock for stimulating discussions, Mathieu Dutour Sikiri\'{c}, Markus Kirschmer, and Rudolf Scharlau for allowing us access to their computations, and the referees for many comments and suggestions improving this text.

\section{Preliminaries}\label{sec:Coxeter}
\subsection{Hyperbolic  space and convex polyhedra}

Fix $n\ge2$, and a bilinear form $f(x,y)$ whose corresponding quadratic form $x\mapsto f(x,x)$ has signature $(n+1, 1)$. This defines a quadratic space $V=\R^{n+1,1}$ with product defined by the bilinear form $f$. We then obtain the disjoint union, $C$, of two open convex cones, 
$$C:=\{x\in V \mid f(x,x) < 0\}.$$ 
We fix some $x_0\in V$ having $f(x_0,x_0)>0$, and define a connected component, $C_+$, of $C$ via: $C_+:=\{x\in C \mid f(x,x_0)>0\}$.
Let $\mathbf{P} C_+ := C_+/\R_+$  denote the  projectivization of $C_+$, which serves as a model for $\HH^{n+1}\cong \mathbf{P} C_+$. We will simplify notation by writing  $(x,y)$ for $f(x,y)$ below.

The points \emph{at infinity} or, in other terms, \emph{on the ideal boundary} $\partial
\HH^{n+1}$, in this model correspond to \emph{isotropic
rays} in $V$,  $\{ \lambda x \mid \lambda \geqslant 0\}$, for $x \in V$ such that $(x,x) = 0$ and $(x,x_0)>0$.

The hyperbolic metric $\rho$ on $\HH^n$ is given by
$$
\cosh \rho (x, y) = -\frac{(x, y)}{\sqrt{(x,x)(y,y)}}.
$$

Let $\OO_{f} (\R)$ be the group of
orthogonal transformations of the space $V$, and let $\OO^\dag_{f}(\R)$ be
its subgroup of index $2$ preserving $\HH^{n+1}$.
The \emph{isometry group} of hyperbolic $(n+1)$--space $\HH^{n+1}$ is 
$\Isom(\HH^{n+1}) \cong \OO^\dag_{f}(\R)$.
By a \emph{lattice}, we mean a discrete subgroup  $\Gamma < \OO^\dag_{f}(\R)$ of finite covolume, i.e. such that the volume of the quotient $\HH^{n+1}/\Gamma$ is finite. A lattice is called \textit{uniform} if $\HH^{n+1}/\Gamma$ is compact, and \emph{non--uniform} otherwise.

Suppose that $e \in V$ has $(e,e)>0$. Then the set
$$
H_{e} = \{x \in \HH^{n+1} \mid (x,e) = 0\}
$$
is a \emph{hyperplane} in $\HH^{n+1}$ and it divides the entire space into the \emph{half--spaces}
$$
H^-_e = \{x \in \HH^{n+1} \mid (x,e) \leqslant 0\}, \qquad H^+_e = \{x \in \HH^{n+1} \mid (x,e) \geqslant 0\}.
$$
The orthogonal transformation given by the formula
$$
R_e(x) = x - 2 \frac{(e, x)}{(e, e)} e,
$$
is called the \emph{reflection in the hyperplane}
$H_e$. The plane $H_e$ 
is called the \emph{mirror} of $R_e$.
This reflection induces an inversion in the ideal boundary in the sphere $S_e = \partial H_e$ which bounds the ball $B_e = \partial H_e^+.$

\begin{definition}\label{def:Pneg}
A convex polyhedron in $\HH^{n+1}$,
$$
P = \bigcap_{j=1}^N H_{e_j}^-,     
$$
is the intersection, assumed to have non--empty interior, of finitely many half--spaces. 
A generalized convex polyhedron is the intersection (with non--empty interior) of a family (possibly, infinite) 
of half--spaces such that every ball intersects only finitely many of their boundary hyperplanes.
\end{definition}

That is, a generalized polyhedron may have infinitely many walls but locally looks like a convex polyhedron.

\begin{definition}
A generalized convex polyhedron is said to be acute--angled if all its dihedral angles do not exceed $\pi/2$. A generalized convex polyhedron is called a Coxeter polyhedron if all its dihedral angles are of the form $\pi/k$, where $k \in \{2,3,4,\ldots,+\infty\}$.
\end{definition}

\begin{definition}
The Coxeter--Vinberg diagram of a Coxeter polyhedron is a graph having a vertex corresponding to each facet, and vertices connected by edges of multiplicity $k$, if the corresponding facets meet at dihedral angle  $\pi/(k+2)$, with $k\geqslant 0$, or by dotted edges. This includes the following interpretations:
\begin{itemize}
    \item An edge of multiplicity $k=\infty$, represented by a bold edge. This happens when the respective facets meet at the ideal boundary $\partial \mathbb{H}^{n+1}$;
    \item An edge of multiplicity $k=0$, i.e. no edge being present between a pair of vertices. This corresponds to the facets being orthogonal;
    \item A dotted edge, corresponding to facets that diverge at infinity and admit a common perpendicular of length $\ell > 0$. Sometimes such an edge may also be labelled with $\cosh \ell$.
\end{itemize}
\end{definition}

It is known that the fundamental domains of discrete reflection groups are generalized Coxeter polyhedra (see \cite{Vin67, Vin85}). A lattice generated by reflections in hyperplanes is called \textit{reflective}.

\subsection{Kleinian sphere packing and the Structure Theorem}

In this subsection we discuss some details regarding Kleinian sphere packings and computational aspects of working with them; see \cite{K^2} for background.

As in \S~2.1, fix a form $f$ acting on a quadratic space $V\cong\R^{n+1,1}$.
Dual to $V$ is the vector space $V^*=\{x^*:V\to\R,\text{ linear}\}$ which is a quadratic space under the dual form $f^*$. The latter is determined as follows. 
Let $x_1^*, x_2^*\in V^*$ be covectors having corresponding vectors $x_1,x_2\in V$, that is,  $x_j^*:x\mapsto f(x,x_j)$ for all $x\in V$; then
 $f^*(x_1^*,x_2^*)=f(x_1,x_2)$. 
 
Next it will be useful and convenient to  construct an ``inversive coordinate system'' for spheres and hyperplanes in the ideal boundary $\partial \HH^{n+1}$ as follows. First, we choose any two $f^*$--isotropic covectors $b, \widehat{b} \in V^*$, such that $f^*(b,\widehat{b}) = -2$. Then take any system of orthonormal covectors $b_j$, $j\in\{1,\dots,n\}$, in the space orthogonal to the plane generated by $b$ and $\widehat{b}$. The inversive coordinate system of a sphere $S_e$ corresponding to the vector $e\in\{x\in V: f(x,x)>0\}$ is then given by 
$$
v_e = ( b(e),\widehat b(e), b_1(e),\dots, b_n(e) ).
$$
This vector $v_e$ contains all of the data of the root $e$, corresponding to a choice of Euclidean coordinates on $\partial\HH^{n+1}$; here $b$ determines a point at infinity, $\widehat b$ determines an origin, and the map $e\mapsto (b_1(e),\dots,b_n(e))/b(e)$, identifies $\{e\in V: f(e,e)>0,\, b(e)\neq 0\}$ with $\R^n\cong \partial \HH^{n+1}\setminus\{\infty\}$. Indeed, consider the sphere $S_e$ in $\partial\HH^{n+1}$ corresponding to the ideal boundary of the hyperplane $H_e$ which is $f$--orthogonal to $e$; then the ``bend'' (inverse signed radius) of the sphere is $b(e)$.
Again, the sign of the radius is determined by setting it to be negative if and only if the point at infinity is contained inside the ball $B_e=\partial H_e^+$.
The co--radius of $S_e$ is defined to be the radius of the inversion of $S_e$ through the unit sphere; the co--bend is the inverse of the co--radius, and is given by $\hat b(e)$. Moreover, the  center of $S_e$ is $(b_1(e),\dots,b_n(e))/b(e).$
If the bend $b(e)=0$, that is, $S_e$ is a hyperplane passing through the point at infinity, then this ``center'' is reinterpreted to mean the limit as spheres converge to this hyperplane; in that case, this ratio becomes the Euclidean unit normal to this hyperplane in $\partial \HH^{n+1}$.

When we have a collection $\{e_\alpha\}_{\alpha \in A}$ of such vectors $e_\alpha\in V$ running over some indexing set $ A$, we simplify notation, writing $S_\alpha$, $H_\alpha$, etc, for $S_{e_\alpha}$, $H_{e_\alpha}$, resp.

\begin{figure}[ht]
    \centering
    \includegraphics[scale=0.3]{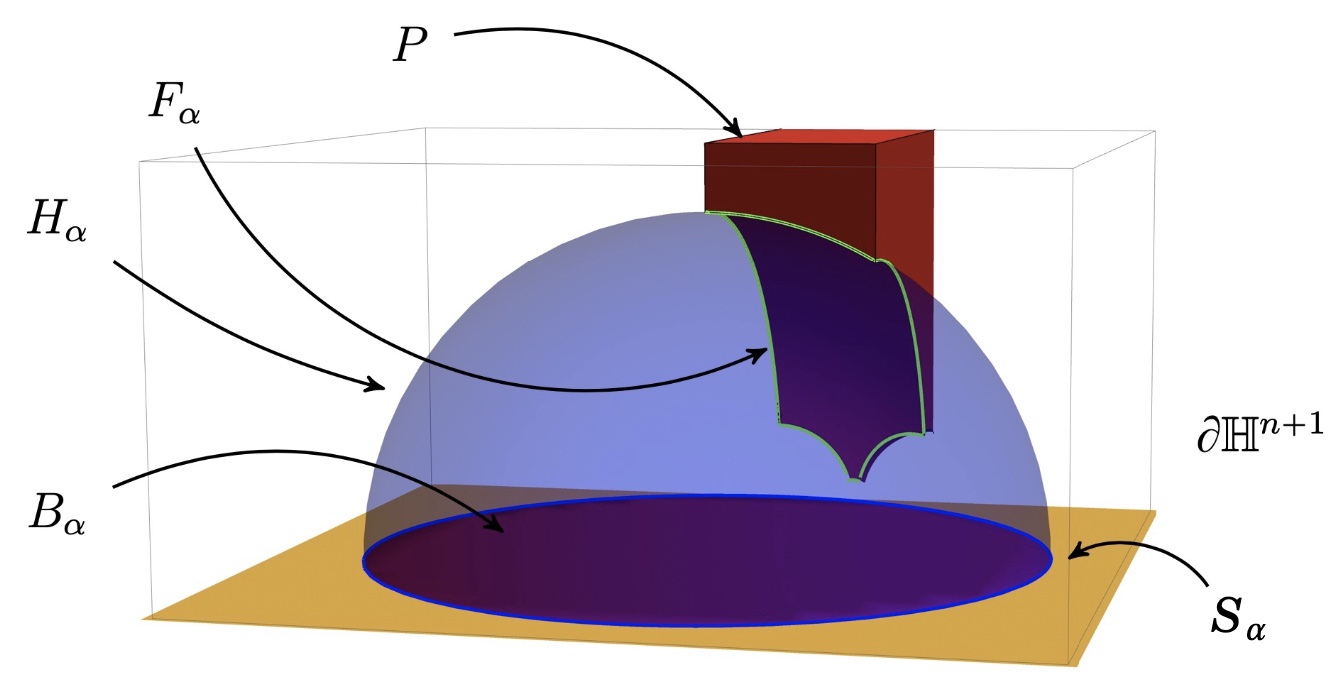}

    \caption{A facet $F_\alpha$ of a polyhedron $P$ has supporting hyperplane $H_\alpha$.}
    \label{fig:struct-th}
\end{figure}

\begin{theorem}[Structure Theorem \cite{KN, K^2}] \label{theorem:structure-kleinian} \quad

\begin{itemize}
\item[\bf (1)] Let $\mathscr{P}$ be a Kleinian packing of $\SSS^n \simeq \partial \HH^{n+1}$. Then its supergroup $\tG$ is a lattice in $\Isom(\HH^{n+1})$.

\item[\bf (2)] Suppose that $\tG < \Isom(\HH^{n+1})$ is a lattice containing at least one reflection with a convex fundamental polyhedron $P$ (as in Definition \ref{def:Pneg}), and a minimal generating set $\widetilde{S}$ for $\tG$, where elements of $\widetilde{S}$ correspond to face--pairing transformations of $P$. Let $\widehat S \subset \widetilde{S}$ be a non--empty set of generators corresponding to reflections $R_\alpha$ in certain facets $F_\alpha$, $\alpha \in A$, of $P$. For each $\alpha \in A$ let
$B_\alpha = \partial_\infty H_\alpha^+\subset \partial \HH^{n+1} = \SSS^n$ be the ball bounded by
$S_\alpha = \partial_\infty H_\alpha$,  the invariant sphere for the reflection $R_\alpha$ with respect to the supporting hyperplane $H_\alpha$ of the facet $F_\alpha$ (see Figure~\ref{fig:struct-th}). Assume that all the hyperplanes $H_\alpha$ are pairwise disjoint or parallel, and that, if they meet other supporting hyperplanes of facets of $P$, then they do so orthogonally. Let $S = \widetilde{S}\setminus \widehat S$ and let $\Gamma = \langle S \rangle$.

Then the orbit of $\{B_\alpha \mid \alpha \in A\}$ under $\Gamma$ is a Kleinian packing of $\SSS^n$ with symmetry group $\Gamma$ and supergroup $\tG$.
\end{itemize}
\end{theorem}

We remark that there is an even simpler description of the above construction in the case of crystallographic sphere packings. Suppose as above that $\widetilde\Gamma$ is generated by a set $\widetilde S$ of reflections, with Coxeter--Vinberg diagram with vertices $\widetilde S$. Suppose $\widetilde S $ can be partitioned into a ``cluster'' $\widehat S$ and ``cocluster'' $S=\widetilde S\setminus \widehat S$, with the following properties: (i) Any pair of vertices of $\widehat S$ are either connected by a dashed edge or a bold edge, and (ii) any vertex of $\widehat S$ and a vertex of $S$ are either connected as above, or disconnected (that is, orthogonal). 
Then, as in the Structure Theorem, the orbit of the spheres in $\widehat S$ under the action of the group $\Gamma $ generated by inversions in $S$ is a crystallographic packing. Moreover, every such packing arises in this way, see \cite{KN}.

\subsection{Vinberg's arithmeticity criterion}
In order to understand if a given lattice generated by reflections is arithmetic or not, we will use a handy criterion developed by Vinberg \cite{Vin67}. Recall that $H_e^- = \{x \in V \mid (x,e) \leqslant 0\}$ is the half--space associated to a hyperplane $H_e$ with normal $e$ having $(e,e)>0$. If $P$ is a Coxeter polyhedron in $\HH^{n+1}$, such that 
$$P = \bigcap_{j=1}^N H_{e_j}^-, $$
then let $G(P) = \{g_{ij}\}^N_{i,j=1} = \{(e_i, e_j)\}^N_{i,j=1}$ denote its Gram matrix, and let the field $K = \Q\left(\{g_{ij}\}^N_{i,j=1}\right)$ be generated by its entries. 

The set of all cyclic products of entries of a matrix $A$ (i.e. the set consisting of all possible products of the form $a_{i_1 i_2} a_{i_2 i_3} \ldots a_{i_k i_1}$) will be denoted by $\Cyc(A)$. Let $k = \Q(\Cyc(G(P))) \subset K$,  and let $\OOO_k$ be the ring of integers of $k$.

Given the lattice $\Gamma$ generated by reflections  in the facets of a Coxeter polyhedron $P$, we can determine if $\Gamma$ is arithmetic, quasi--arithmetic, or neither.

\begin{theorem}[Vinberg's arithmeticity criterion \cite{Vin67}]\label{V}
Let $\Gamma$ be a reflection group acting on $\HH^{n+1}$ with finite volume fundamental Coxeter polyhedron $P$. Then $\Gamma$ is arithmetic if and only if
the following conditions hold:
\begin{itemize}
    \item[{\bf(1)}] $K$ is a totally real algebraic number field;
    \item[{\bf(2)}] for any embedding $\sigma \colon K \to \R$, such that $\sigma \mid_{k} \ne \mathrm{id}$, $G^\sigma(P)$ is positive semidefinite;
    \item[{\bf(3)}] $\Cyc(2 \cdot G(P)) \subset \OOO_{k}$.
\end{itemize}
The group $\Gamma$ is  quasi--arithmetic if and only if it satisfies conditions \textbf{(1)}--\textbf{(2)}, but not necessarily \textbf{(3)}.
\end{theorem}

When $\Gamma$ is (quasi)arithmetic, $k$ is called its \textit{field of definition} (or \textit{ground field}), and we say that $\Gamma$ is $k$--(quasi)arithmetic to emphasize the ground field. The above criterion simplifies greatly if $k = \mathbb{Q}$. Then the lattice $\Gamma$ will always be quasi--arithmetic, and it will be arithmetic if and only if the cyclic products in Theorem~\ref{V}~(3) are integers. 

\subsection{Vinberg's algorithm} In addition to his study of arithmeticity of reflective lattices, Vinberg also devised an algorithm that produces the set of generating reflections for the maximal reflective subgroup of the automorphism group of a quadratic from. 

Let $f$ be a quadratic form of signature $(n+1, 1)$ over a totally real number field $k$ with the ring of integers $\OOO_k$. Moreover, assume that $f$ is admissible, that is, for each non--trivial Galois embedding $\sigma: k \rightarrow \mathbb{R}$ the form $f^\sigma$ (obtained by applying $\sigma$ to the coefficients of $f$) is positive definite. Then the group $\OO_f(\OOO_k)$ of integer points of the orthogonal group $\OO_f(\mathbb{R})$ is a lattice \cite[Chapter 6]{AVS93}. 

The form $f$ as above is called reflective if $\OO_f(\OOO_k)$ is generated by a finite number of reflections, up to finite index. The algorithm devised by Vinberg in \cite{Vin72} produces a set of generating reflections for any reflective quadratic form. We shall refer to it as \textit{the Vinberg algorithm}. Several computer realizations of the algorithm are available \cite{Gug, BP, Bot}.

In contrast, if the form $f$ is not reflective, the algorithm never halts. The later work by Bugaenko \cite{Bug} also provided a way to certify that a given Lorentzian form $f$ is non--reflective. This ``method of infinite symmetry'' has been also implemented in \cite{Gug, BP, Bot} to various extent. 

\section{Proofs}
\label{sec:proofs}

\subsection{Proof of Lemma~\ref{lemma:doubling}}

Let $P$ be a finite volume Coxeter polyhedron in $\HH^{n+1}$ and $P_0$ be a facet that meets other facets of $P$ at angles either $\pi/4$ or $\pi/2$. Let $P_1, \ldots P_k$ be all the facets of $P$ that meet $P_0$ at $\pi/4$. Assume that each $P_j$ meets its neighbors at ``even'' angles of the form $\pi/2^m$, with $m \geqslant 1$, an integer. 

If $P$ is a polyhedron as above, and $F$ is a facet of $P$, let $R_F$ denote the reflection in the hyperplane of $F$. 

Then the following inductive steps will produce a polyhedron with a facet orthogonal to its neighbors.

1. The double $P^{(1)}$ of $P$ along the facet $F_1 = P_1$ has facets $P_0$ and $R_{F_1} (P_0)$, both of which form $(k-1)$ angles $\pi/4$ with their respective neighbours. Observe that $P^{(1)}$ is again a Coxeter polyhedron, since doubling of even Coxeter angles produces even Coxeter angles again or, if we double a right angle, it produces a new facet.  

2. The facet $F_2 = R_{F_1} (P_2)$ of $P^{(1)}$ forms even Coxeter angles with its neighbours. Thus the double $P^{(2)}$ of $P^{(1)}$ along $F_2$ has the facet $R_{F_1}(P_0)$ that forms $(k-2)$ angles $\pi/4$ with its neighbours. All the conclusions of the previous step apply respectively to $P^{(2)}$.  

3. Continuing in this way, we obtain a polyhedron  $P^{(k)}$ from  $P^{(k-1)}$ by doubling along  $F_k = R_{F_{k-1}} (P_k)$, and the facet 
$R_{F_{k-1}}\cdots R_{F_1}(P_0)$ of $P^{(k)}$ will be orthogonal to all of its neighbours. Such a polyhedron produces a crystallographic sphere packing by the Structure Theorem (Theorem~\ref{theorem:structure-kleinian}). 

\subsection{Proof of Theorem~\ref{theorem:dimension-bound}}\label{sec:A}

The proof of Theorem~\ref{theorem:dimension-bound} follows from the results of Esselmann \cite{E}, the Structure Theorem, and the following simple lemma which is well known to experts in hyperbolic reflection groups (and also follows from more general results of \cite{BK21}). 

\begin{lemma}\label{lemma:stabilizer}
Let $\Gamma<\Isom(\HH^{n+1})$ be the symmetry group of a crystallographic packing $\mathscr{P}$ of $\SSS^n$, and let $\Gamma_0$ be the stabilizer of a sphere in $\mathscr{P}$. Then $\Gamma_0$ is a reflective lattice acting on $\HH^n$. If $\Gamma$ is arithmetic, then $\Gamma_0$ is also arithmetic.
\end{lemma}

\begin{proof}
Indeed, the sphere corresponds to a facet $P_0$ of the Coxeter polyhedron $P$ associated to the packing $\mathscr{P}$, and $P_0$ is orthogonal to all the adjacent facets. Then $\Gamma_0$ coincides with the finite covolume hyperbolic group generated by reflections in the facets of $P_0$. The orthogonality of $P_0$ to its neighbours together with Vinberg's criterion (Theorem~\ref{V}) implies that $\Gamma$ being arithmetic makes $\Gamma_0$ arithmetic too.  
\end{proof}

\medskip

It is easy to see that Lemma~\ref{lemma:stabilizer} implies Proposition~\ref{prop:lower-dimension} by applying it to the setting of superintegral crystallographic sphere packings. 
Indeed, in \cite{E}, Esselmann showed that Borcherds' $21$--dimensional polyhedron from \cite{Bor} is a unique, up to commensurablity, $\Q$--arithmetic Coxeter polyhedron in $\HH^{21}$ and that there are no $\Q$--arithmetic Coxeter polyhedra in $\HH^{20}$. 
If there was a superintegral crystallographic  packing of $\SSS^{19}$, the Structure Theorem would imply the existence of a $\Q$--arithmetic reflective lattice in $\HH^{20}$; this is ruled out by Esselmann.
Now, if there was a superintegral crystallographic packing of $\SSS^{20}$ coming from a $\Q$--arithmetic reflective lattice in $\HH^{21}$, then there would be a $\Q$--arithmetic reflective lattice in $\HH^{20}$ by Lemma \ref{lemma:stabilizer}; this is again ruled out by Esselmann \cite{E}. Similarly, any superintegral crystallographic packing of $\SSS^n$ with $n\geqslant 21$ is ruled out by Esselmann by the nonexistence of $\Q$--arithmetic reflective lattices in these dimensions.

\subsection{Proof of Theorem~\ref{theorem:crystallographic}}\label{sec:B} Our higher--dimensional examples of superintegral crystallographic sphere packings come from  Borcherds's $21$--di\-men\-sio\-nal polyhedron $P_{21}$ \cite{Bor}. By describing its lower--dimensional faces we can fill in some of the dimensions. For example, in $\mathbb{H}^{19}$ we have a face $P_{19}$ (first found by Vinberg  and Kaplinskaja in \cite{KV}) that satisfies the conditions of Lemma~\ref{lemma:doubling}.

There are two faces $P^{(i)}_{18}$, $i=1,2$, each with an orthogonal facet, in $\mathbb{H}^{18}$. Here, $P^{(1)}_{18}$ denotes the polyhedron (with $37$ facets) described by  Vinberg and Kaplinskaja in \cite{KV}. The other polyhedron (with $24$ facets) $P^{(2)}_{18}$ was found by Allcock in \cite[Examples 2.4 \& 2.5]{A}. Both polyhedra $P^{(1)}_{18}$ and $P^{(2)}_{18}$ are also described by Vinberg in \cite[Section 7]{Vin15}. 

By descending to lower--dimensional Coxeter faces (sometimes we have to pass via non--Coxeter faces of lower codimension in order to find their faces of higher codimension that turn out to be Coxeter), we obtain another face $P_{15}$ satisfying Lemma~\ref{lemma:doubling} in $\mathbb{H}^{15}$. 

The matrix provided in Figure~\ref{fig:P15-matrix} describes the combinatorial and geometric structure of $P_{15}$. Notice that we can apply Lemma~\ref{lemma:doubling} to facet $1$, as follows from the first three rows (or columns) of the given matrix. 

\begin{figure}
\centering
\begin{footnotesize}
$$
\left(
\begin{array}{ccccccccccccccccccccccccccccc}
 {\scriptscriptstyle \blacklozenge} & 4 & 4 & \text{u} & \text{u} & \text{u} & \text{u} &
   \cdot & \text{u} & \cdot & \cdot & \cdot & \cdot &
   \text{u} & \cdot & \cdot & \cdot & \cdot &
   {\scriptscriptstyle\infty} & \cdot & \cdot & \cdot & \cdot &
   {\scriptscriptstyle\infty} & \cdot & {\scriptscriptstyle\infty} & \cdot & \cdot &
   \cdot \\
 4 & {\scriptscriptstyle \blacklozenge} & {\scriptscriptstyle\infty} & 4 & 4 & 4 & 4 & \cdot & 4 &
   \cdot & \cdot & \cdot & \cdot & \cdot & \cdot &
   \cdot & \cdot & \cdot & 4 & \cdot & \cdot &
   \cdot & 4 & \cdot & \cdot & \cdot & \cdot & 4 &
   \cdot \\
 4 & {\scriptscriptstyle\infty} & {\scriptscriptstyle \blacklozenge} & 4 & 4 & 4 & 4 & \cdot & 4 &
   \cdot & \cdot & \cdot & \cdot & \cdot & \cdot &
   \cdot & \cdot & \cdot & \cdot & 4 & \cdot &
   \cdot & \cdot & 4 & \cdot & \cdot & 4 & \cdot &
   \cdot \\
 \text{u} & 4 & 4 & {\scriptscriptstyle \blacklozenge} & \text{u} & \text{u} & \text{u} &
   \cdot & \text{u} & \cdot & {\scriptscriptstyle\infty} & \cdot &
   \cdot & \cdot & \cdot & \text{u} & \cdot & \cdot &
   \cdot & \cdot & \cdot & \cdot & \cdot & \cdot &
   \cdot & \cdot & {\scriptscriptstyle\infty} & {\scriptscriptstyle\infty} & \cdot \\
 \text{u} & 4 & 4 & \text{u} & {\scriptscriptstyle \blacklozenge} & \text{u} & \text{u} &
   \cdot & \text{u} & {\scriptscriptstyle\infty} & \cdot & \cdot &
   \cdot & \cdot & \cdot & \text{u} & \cdot & \cdot &
   \cdot & \cdot & \cdot & \cdot & {\scriptscriptstyle\infty} &
   {\scriptscriptstyle\infty} & \cdot & \cdot & \cdot & \cdot &
   \cdot \\
 \text{u} & 4 & 4 & \text{u} & \text{u} & {\scriptscriptstyle \blacklozenge} & \text{u} &
   \cdot & \text{u} & \cdot & \cdot & \cdot &
   {\scriptscriptstyle\infty} & \text{u} & \cdot & \cdot & \cdot &
   \cdot & \cdot & {\scriptscriptstyle\infty} & \cdot & \cdot &
   \cdot & \cdot & \cdot & \cdot & \cdot &
   {\scriptscriptstyle\infty} & \cdot \\
 \text{u} & 4 & 4 & \text{u} & \text{u} & \text{u} & {\scriptscriptstyle \blacklozenge} &
   \cdot & \text{u} & \cdot & \cdot & \cdot & \cdot &
   \cdot & \cdot & \cdot & \cdot & \cdot & \cdot &
   {\scriptscriptstyle\infty} & \cdot & \text{u} & {\scriptscriptstyle\infty} & \cdot &
   {\scriptscriptstyle\infty} & \cdot & \cdot & \cdot & \cdot \\
 \cdot & \cdot & \cdot & \cdot & \cdot & \cdot &
   \cdot & {\scriptscriptstyle \blacklozenge} & \cdot & \cdot & \cdot & \cdot
   & \cdot & \cdot & \cdot & \cdot & \cdot & \cdot
   & \cdot & 3 & 3 & \cdot & \cdot & \cdot & \cdot &
   3 & \cdot & \cdot & \cdot \\
 \text{u} & 4 & 4 & \text{u} & \text{u} & \text{u} & \text{u} &
   \cdot & {\scriptscriptstyle \blacklozenge} & \cdot & \cdot & \cdot & \cdot
   & \cdot & \cdot & \cdot & \cdot & \cdot &
   {\scriptscriptstyle\infty} & \cdot & {\scriptscriptstyle\infty} & \text{u} & \cdot &
   \cdot & \cdot & \cdot & {\scriptscriptstyle\infty} & \cdot &
   \cdot \\
 \cdot & \cdot & \cdot & \cdot & {\scriptscriptstyle\infty} & \cdot
   & \cdot & \cdot & \cdot & {\scriptscriptstyle \blacklozenge} & \cdot &
   \cdot & \cdot & \cdot & 3 & \cdot & \cdot &
   \cdot & \cdot & \cdot & \cdot & \cdot & \cdot &
   \cdot & \cdot & \cdot & \cdot & \cdot & 3 \\
 \cdot & \cdot & \cdot & {\scriptscriptstyle\infty} & \cdot & \cdot
   & \cdot & \cdot & \cdot & \cdot & {\scriptscriptstyle \blacklozenge} &
   \cdot & \cdot & \cdot & \cdot & \cdot & 3 & 3 &
   \cdot & \cdot & \cdot & \cdot & \cdot & \cdot &
   \cdot & \cdot & \cdot & \cdot & \cdot \\
 \cdot & \cdot & \cdot & \cdot & \cdot & \cdot &
   \cdot & \cdot & \cdot & \cdot & \cdot & {\scriptscriptstyle \blacklozenge}
   & 3 & \cdot & \cdot & \cdot & \cdot & \cdot & 3 &
   \cdot & \cdot & \cdot & \cdot & \cdot & 3 &
   \cdot & \cdot & \cdot & \cdot \\
 \cdot & \cdot & \cdot & \cdot & \cdot & {\scriptscriptstyle\infty}
   & \cdot & \cdot & \cdot & \cdot & \cdot & 3 &
   {\scriptscriptstyle \blacklozenge} & \cdot & \cdot & \cdot & \cdot & 3 &
   \cdot & \cdot & \cdot & \cdot & \cdot & \cdot &
   \cdot & \cdot & \cdot & \cdot & \cdot \\
 \text{u} & \cdot & \cdot & \cdot & \cdot & \text{u} &
   \cdot & \cdot & \cdot & \cdot & \cdot & \cdot &
   \cdot & {\scriptscriptstyle \blacklozenge} & \cdot & 3 & \cdot & \cdot &
   \cdot & \cdot & \cdot & 3 & 4 & \cdot & \cdot &
   \cdot & 4 & \cdot & \cdot \\
 \cdot & \cdot & \cdot & \cdot & \cdot & \cdot &
   \cdot & \cdot & \cdot & 3 & \cdot & \cdot &
   \cdot & \cdot & {\scriptscriptstyle \blacklozenge} & \cdot & \cdot & \cdot
   & \cdot & \cdot & \cdot & \cdot & \cdot & \cdot
   & 3 & \cdot & 3 & \cdot & \cdot \\
 \cdot & \cdot & \cdot & \text{u} & \text{u} & \cdot &
   \cdot & \cdot & \cdot & \cdot & \cdot & \cdot &
   \cdot & 3 & \cdot & {\scriptscriptstyle \blacklozenge} & \cdot & \cdot & 4 &
   4 & \cdot & 3 & \cdot & \cdot & \cdot & \cdot &
   \cdot & \cdot & \cdot \\
 \cdot & \cdot & \cdot & \cdot & \cdot & \cdot &
   \cdot & \cdot & \cdot & \cdot & 3 & \cdot &
   \cdot & \cdot & \cdot & \cdot & {\scriptscriptstyle \blacklozenge} & \cdot
   & \cdot & \cdot & 3 & \cdot & 3 & \cdot & \cdot &
   \cdot & \cdot & \cdot & \cdot \\
 \cdot & \cdot & \cdot & \cdot & \cdot & \cdot &
   \cdot & \cdot & \cdot & \cdot & 3 & \cdot & 3 &
   \cdot & \cdot & \cdot & \cdot & {\scriptscriptstyle \blacklozenge} & \cdot
   & \cdot & \cdot & \cdot & \cdot & 3 & \cdot &
   \cdot & \cdot & \cdot & \cdot \\
 {\scriptscriptstyle\infty} & 4 & \cdot & \cdot & \cdot & \cdot &
   \cdot & \cdot & {\scriptscriptstyle\infty} & \cdot & \cdot & 3 &
   \cdot & \cdot & \cdot & 4 & \cdot & \cdot &
   {\scriptscriptstyle \blacklozenge} & \cdot & \cdot & \cdot & \cdot & \cdot
   & \cdot & \cdot & \cdot & \cdot & \cdot \\
 \cdot & \cdot & 4 & \cdot & \cdot & {\scriptscriptstyle\infty} &
   {\scriptscriptstyle\infty} & 3 & \cdot & \cdot & \cdot & \cdot &
   \cdot & \cdot & \cdot & 4 & \cdot & \cdot &
   \cdot & {\scriptscriptstyle \blacklozenge} & \cdot & \cdot & \cdot & \cdot
   & \cdot & \cdot & \cdot & \cdot & \cdot \\
 \cdot & \cdot & \cdot & \cdot & \cdot & \cdot &
   \cdot & 3 & {\scriptscriptstyle\infty} & \cdot & \cdot & \cdot &
   \cdot & \cdot & \cdot & \cdot & 3 & \cdot &
   \cdot & \cdot & {\scriptscriptstyle \blacklozenge} & \cdot & \cdot & \cdot
   & \cdot & \cdot & \cdot & \cdot & \cdot \\
 \cdot & \cdot & \cdot & \cdot & \cdot & \cdot &
   \text{u} & \cdot & \text{u} & \cdot & \cdot & \cdot &
   \cdot & 3 & \cdot & 3 & \cdot & \cdot & \cdot &
   \cdot & \cdot & {\scriptscriptstyle \blacklozenge} & \cdot & 4 & \cdot &
   \cdot & \cdot & 4 & \cdot \\
 \cdot & 4 & \cdot & \cdot & {\scriptscriptstyle\infty} & \cdot &
   {\scriptscriptstyle\infty} & \cdot & \cdot & \cdot & \cdot &
   \cdot & \cdot & 4 & \cdot & \cdot & 3 & \cdot &
   \cdot & \cdot & \cdot & \cdot & {\scriptscriptstyle \blacklozenge} & \cdot
   & \cdot & \cdot & \cdot & \cdot & \cdot \\
 {\scriptscriptstyle\infty} & \cdot & 4 & \cdot & {\scriptscriptstyle\infty} & \cdot &
   \cdot & \cdot & \cdot & \cdot & \cdot & \cdot &
   \cdot & \cdot & \cdot & \cdot & \cdot & 3 &
   \cdot & \cdot & \cdot & 4 & \cdot & {\scriptscriptstyle \blacklozenge} &
   \cdot & \cdot & \cdot & \cdot & \cdot \\
 \cdot & \cdot & \cdot & \cdot & \cdot & \cdot &
   {\scriptscriptstyle\infty} & \cdot & \cdot & \cdot & \cdot & 3 &
   \cdot & \cdot & 3 & \cdot & \cdot & \cdot &
   \cdot & \cdot & \cdot & \cdot & \cdot & \cdot &
   {\scriptscriptstyle \blacklozenge} & \cdot & \cdot & \cdot & \cdot \\
 {\scriptscriptstyle\infty} & \cdot & \cdot & \cdot & \cdot & \cdot
   & \cdot & 3 & \cdot & \cdot & \cdot & \cdot &
   \cdot & \cdot & \cdot & \cdot & \cdot & \cdot &
   \cdot & \cdot & \cdot & \cdot & \cdot & \cdot &
   \cdot & {\scriptscriptstyle \blacklozenge} & \cdot & \cdot & 3 \\
 \cdot & \cdot & 4 & {\scriptscriptstyle\infty} & \cdot & \cdot &
   \cdot & \cdot & {\scriptscriptstyle\infty} & \cdot & \cdot &
   \cdot & \cdot & 4 & 3 & \cdot & \cdot & \cdot &
   \cdot & \cdot & \cdot & \cdot & \cdot & \cdot &
   \cdot & \cdot & {\scriptscriptstyle \blacklozenge} & \cdot & \cdot \\
 \cdot & 4 & \cdot & {\scriptscriptstyle\infty} & \cdot & {\scriptscriptstyle\infty} &
   \cdot & \cdot & \cdot & \cdot & \cdot & \cdot &
   \cdot & \cdot & \cdot & \cdot & \cdot & \cdot &
   \cdot & \cdot & \cdot & 4 & \cdot & \cdot &
   \cdot & \cdot & \cdot & {\scriptscriptstyle \blacklozenge} & 3 \\
 \cdot & \cdot & \cdot & \cdot & \cdot & \cdot &
   \cdot & \cdot & \cdot & 3 & \cdot & \cdot &
   \cdot & \cdot & \cdot & \cdot & \cdot & \cdot &
   \cdot & \cdot & \cdot & \cdot & \cdot & \cdot &
   \cdot & 3 & \cdot & 3 & {\scriptscriptstyle \blacklozenge} \\
\end{array}
\right)
$$
\end{footnotesize}
\caption{The polyhedron $P_{15}$ described by its adjacency matrix: each off--diagonal $(i,j)$--entry represents either the dihedral angle at which facet $i$ intersects facet $j$ (marked with ``$3$'' for a $\pi/3$ angle, ``$4$'' for a $\pi/4$ angle, and marked with a dot for a right angle), or a pair of parallel (marked with ``$\infty$'') or ultraparalell (marked with ``$\text{u}$'') facets.}\label{fig:P15-matrix}
\end{figure}

The polyhedron $P_{15}$ has $29$ facets and generates an arithmetic reflection group, as follows from Vinberg's arithmeticity criterion (Theorem~\ref{V}). Indeed, conditions (1) and (2) are vacuously satisfied. The only condition that needs to be checked is (3). As could be easily observed, if $G = (g_{ij})^{29}_{i,j=1}$ is the Gram matrix of $P_{15}$, then each $g_{ij}$ satisfies $4\,g^2_{ij} \in \mathbb{Z}$, and thus the elements of $\mathrm{Cyc}(2G)$ are either all integers, or at least one of them is a quadratic irrationality. In order to exclude the latter, it is enough to check that $2^k g_{i_1 i_2} \ldots g_{i_{k-1} i_{k}} g_{i_k i_1} \in \mathbb{Z}$, for $i_1 i_2 \ldots i_k i_1$ being edge cycles in the Coxeter diagram of $P_{15}$ belonging to a $\mathbb{Z}_2$--cycle basis of the diagram (viewed as a simple undirected graph). Indeed, if a cycle product $p$ corresponds to a cycle $c$ that is $\mathbb{Z}_2$--sum of $c_1$ and $c_2$ (with the corresponding cycle products $p_1$ and $p_2$) then $p_1 \cdot p_2 = p \cdot q^2$, where $q$ corresponds to the overlap of $c_1$ and $c_2$ (i.e. $q$ is a trivial back--and--forth cycle unless $c_1 = c_2$ as sets, which is a trivial situation). As mentioned above $q^2 \in \mathbb{Z}$, and thus $p$ cannot be a quadratic irrationality if $p_1, p_2 \in \mathbb{Z}$. This argument easily generalizes to an arbitrary $\mathbb{Z}_2$--sum of cycles.

Let us observe that there are no suitable Coxeter faces of $P_{21}$ in $\HH^{n+1}$ for $n = 15$ and $16$ in order to apply Lemma~\ref{lemma:doubling}.

The remaining packings in $\SSS^n$, $2 \leqslant n \leqslant 13$, can be easily obtained from Vinberg's examples \cite{Vin72} for the series of quadratic forms~$f_n = -2 x^2_0 + x^2_1 + \ldots + x^2_n$, for $3 \leqslant n \leqslant 14$, respectively. Each of the polyhedra in \cite[Table 7]{Vin72} has a facet that is orthogonal to its neighbours\footnote{For packings in $\mathbb{S}^n$, $2 \leqslant n \leqslant 7$, one can also use the totally right--angled polyhedra from \cite{PV}.}. 

All the computations mentioned above (faces of Borcherds' $P_{21}$ and their properties of being Coxeter and/or having ``even'' angles) are accessible on Github \cite{github1} as a SageMath \cite{sagemath} worksheet. 

\subsection{Proof of Theorem~\ref{theorem:arithmetic-super}}\label{sec:C} In \cite{SW}, Scharlau and Walhorn provide the list of square--free isotropic reflective Lorentzian lattices\footnote{A \emph{Lorentzian lattice} is a free $\Z$-module equipped with a bilinear product of signature $(n+1, 1)$. This should not be confused with lattices in the sense of discrete subgroups of hyperbolic isometries.} of signature $(n+1, 1)$ for $n = 2, 3$. In \cite{Turkalj}, Turkalj performs an analogous classification for $n = 4$. They correspond (up to commensurability) to maximal arithmetic non--uniform reflective lattices acting on $\HH^{n+1}$, for $n = 2, 3, 4$.

It is worth mentioning that the list of lattices acting on $\mathbb{H}^4$ has a few entries where the $\mathbb{H}$ direct summand apparently has to be replaced by ${}^2 \mathbb{H}$; see our Git repository \cite{github1} for details. Such entries are easily identifiable by running the Vinberg algorithm (the number of roots indicated in the respective tables seems to be correct). The lattices described by Turkalj require more work for the information to be extracted from the list provided in \cite{Turkalj}. We used instead an equivalent list provided by Kirschmer and Scharlau \cite{KS}. 

In each case we reduce each list up to rational isometry between the respective lattices. Then the proof amounts to a lengthy computation carried out by using the Julia \cite{julia} implementation of Vinberg's algorithm due to Bottinelli \cite{Bot}. 

For each lattice, we consider its reflection subgroup: if the corresponding polyhedron has an orthogonal facet (or, in terms of Vinberg's work \cite{Vin15}, the lattice has an \textit{isolated root}), then we are done. Otherwise, we choose another lattice, up to an appropriate rational isometry, and use it instead: we get a commensurable lattice having an isolated root. The necessary rational isometry is usually guessed by trial and error: reflectivity is checked by using the Vinberg algorithm. 

All the data necessary to verify our computation \cite{github1} and Bottinelli's implementation of the Vinberg algorithm \cite{github2} are available on Github.

\subsection{A remark on reflectivity}\label{sec:remark-refl}

Here we consider another question in the context of properly integral  packings. 

In the proof of the Arithmeticity Theorem (see \cite[Theorem 19]{KN}), given an integral orbit of balls $\mathcal O = B_0 \cdot \Gamma$, and a choice of $n+2$ linearly independent normals to spheres bounding the balls in the orbit,  one constructs a Lorentzian lattice $L$ and its corresponding arithmetic lattice $O_L(\Z)$, so that $O_L(\Z)$ contains a conjugate of $\Gamma$  as a subgroup; therefore $\Gamma$ is subarithmetic, in the terminology of \cite[p. 3]{K^2}. 

\begin{figure}[h]
    \centering
 \begin{tikzpicture}[scale=1.5]
    \coordinate (one) at (0, 0);
    \coordinate (two) at (1,1);
    \coordinate (three) at (1,-1);
    \coordinate (four) at (2,0);
    \coordinate (five) at (3,0);
    \coordinate (midpt) at (4,0);
    \coordinate (six) at (5,0);

    \draw[dashed] (five) -- (six) ;
    \draw[thick] (five) -- (four);
    \draw[thick, double distance = 2pt] (four) -- (three);
    \draw[thick, double distance = 2pt] (one) -- (two);
    \draw[thick] (one) -- (three);
    \draw[thick] (two) -- (four);
    
    \filldraw (midpt) circle (0pt) node[above] { $\sqrt{\nicefrac{8\, }7}$};
    \filldraw (one) circle (2pt) node[left] {\scriptsize 1};
    \filldraw (two) circle (2pt) node[above] {\scriptsize 2};
    \filldraw (three) circle (2pt) node[below] {\scriptsize 3} ;
    \filldraw (four) circle (2pt) node[above] {\scriptsize 4};
    \filldraw (five) circle (2pt) node[above] {\scriptsize 5};
    \filldraw (six) circle (2pt) node[above] {\scriptsize 6};
    
  \end{tikzpicture}
    \caption{The Coxeter--Vinberg diagram of the prism $P$.}
    \label{fig:P}
\end{figure}
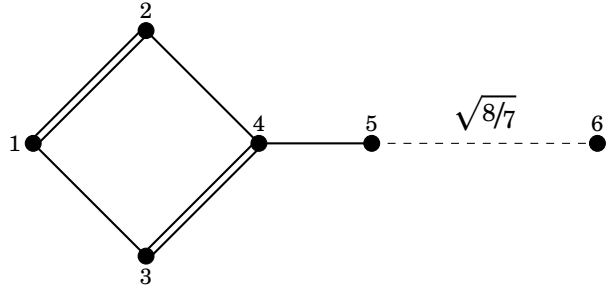

One may wonder, since $\Gamma$ is generated by reflections, whether the Lorentzian lattice $L$ should itself be reflective. Here we record an example which answers this question in the negative.

Let $P$ be the $4$--dimensional non--compact hyperbolic prism described by the  Coxeter--Vinberg diagram in Figure \ref{fig:P}.
The facet normals of $P$ are
\begin{equation*}
    \begin{aligned}[t]
    &e_1 = \left(-\frac{\sqrt{2}}{8}, \sqrt{2}, 0, \frac{1}{2}, \frac{\sqrt{2}}{2}\right),\\
    &\\
    &e_3 = (0, 0, 0, -1, 0),\\
    &\\
    &e_5 = (0, 0, -1, 0, 0),\\
    \end{aligned}
    \qquad
    \begin{aligned}[t]
    &e_2 = (0, 0, 0, 0, -1),\\
    &\\
    &e_4 = \left(1, 0, \frac{1}{2}, \frac{\sqrt{2}}{2}, \frac{1}{2}\right),\\
    &e_6=\left(\frac{\sqrt{14}}{28}, \frac{2 \sqrt{14}}{7}, \frac{2 \sqrt{14}}{7}, 0, 0\right),
    \end{aligned}
\end{equation*}
where the isolated vector is $e_6$ and the scalar product is given by the matrix
$$
\begin{pmatrix}
    0 & -\frac{1}{2} & 0 & 0 & 0 \\
    -\frac{1}{2} & 0 & 0 & 0 & 0 \\
    0 & 0 & -1 & 0 & 0 \\
    0 & 0 & 0 & -1 & 0 \\
    0 & 0 & 0 & 0 & -1
\end{pmatrix}.
$$
By applying Vinberg's criterion (Theorem~\ref{V}) to the Gram matrix of $P$
\begin{equation*}
    G(P) = 
    \left(
\begin{array}{cccccc}
 -1 & \frac{1}{\sqrt{2}} & \frac{1}{2} & 0 & 0 & 0 \\
 \frac{1}{\sqrt{2}} & -1 & 0 & \frac{1}{2} & 0 & 0 \\
 \frac{1}{2} & 0 & -1 & \frac{1}{\sqrt{2}} & 0 & 0 \\
 0 & \frac{1}{2} & \frac{1}{\sqrt{2}} & -1 & \frac{1}{2} & 0 \\
 0 & 0 & 0 & \frac{1}{2} & -1 &  \sqrt{\frac{8}{7}} \\
 0 & 0 & 0 & 0 &  \sqrt{\frac{8}{7}} & -1 \\
\end{array}
\right)
\end{equation*}
we can conclude that $P$ generates a quasi--arithmetic but not arithmetic reflection group. Hence the sphere packing corresponding to $e_6$ cannot be superintegral.

\begin{figure}
    \centering
    \includegraphics[width=.6\textwidth]{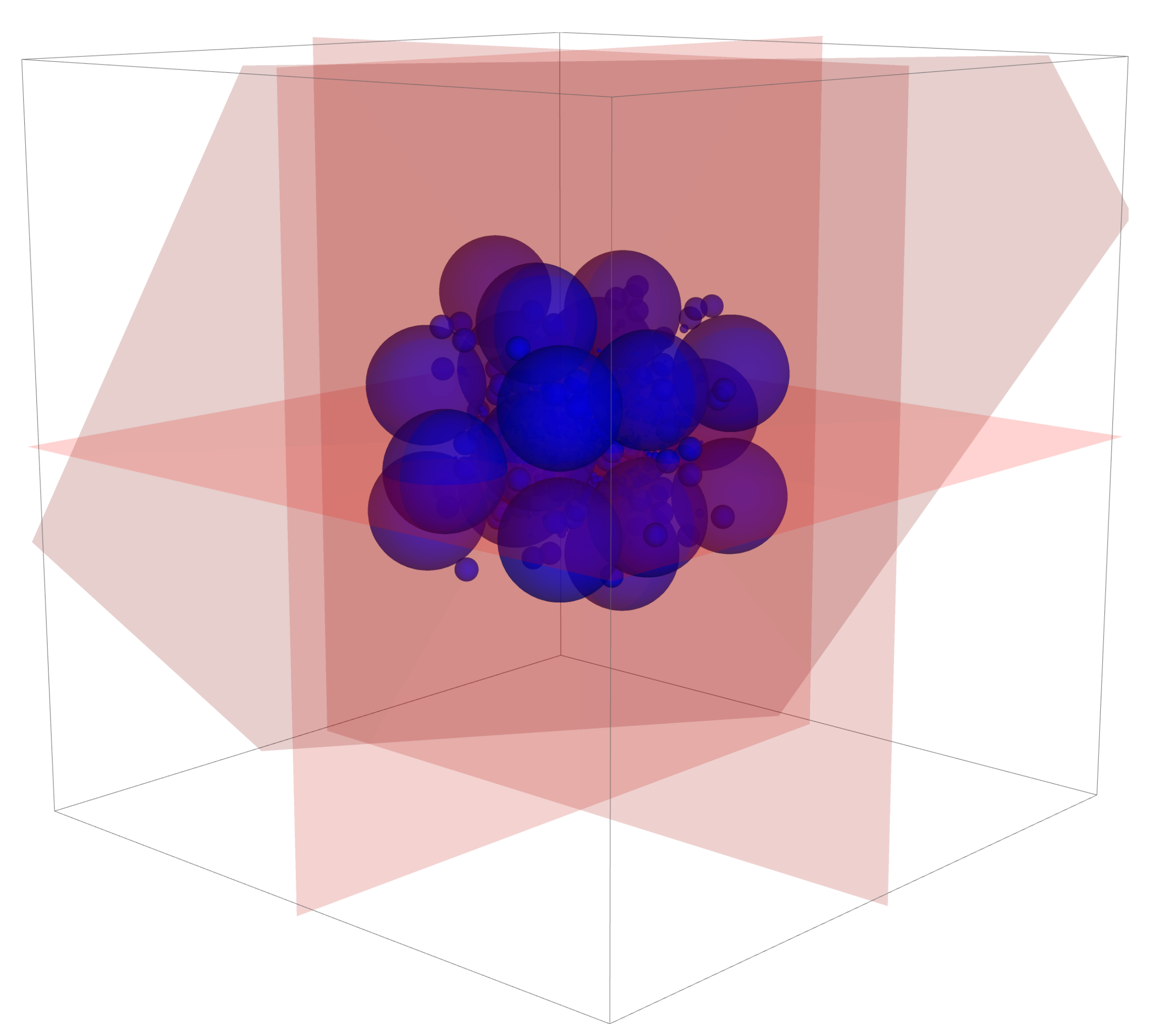}
    \caption{Some of the balls of the packing $\mathscr{P}$ shown in blue, and the walls of the generators $e_1,\dots,e_5$ of $\Gamma$ shown in red.}
    \label{fig:pack}
\end{figure}

The orbit of the sphere normal to $e_6$ under the group $\Gamma$ generated by reflections through the normals $e_1,\dots,e_5$ is a crystallographic  packing (see Figure \ref{fig:pack}), which we claim is properly integral.
To see this, we choose for our list of $n+2=5$ linearly independent vectors the following: 
$$V = \left( e_6, e_6 R_5, e_6 R_5 R_4, e_6 R_5 R_4 R_2, e_6 R_5 R_4 R_3 \right),$$ 
as described in \cite[Equation (2)]{KN}.
These are all in the packing $\mathscr{P} = B_{ e_6}\cdot\Gamma$.
After rescaling by $\sqrt {8/7}$, these balls all have bend $1$. 

Then the generators $R_1,\dots,R_5$, conjugated by $V$, become the matrices:
$$
VR_1V^{-1}=
\left(
\begin{array}{ccccc}
 1 & 0 & 0 & 0 & 0 \\
 0 & 1 & 0 & 0 & 0 \\
 0 & 0 & 1 & 0 & 0 \\
 9/4 & 9/4 & 3/2 & -1/4 & \
-3/4 \\
 9/4 & 9/4 & 3/2 & -5/4 & 1/4 \
\\
\end{array}
\right),
\
VR_2V^{-1}=
\left(
\begin{array}{ccccc}
 1 & 0 & 0 & 0 & 0 \\
 0 & 1 & 0 & 0 & 0 \\
 0 & 0 & 0 & 1 & 0 \\
 0 & 0 & 1 & 0 & 0 \\
 0 & 0 & -1 & 1 & 1 \\
\end{array}
\right),
$$
$$
VR_3V^{-1}=
\left(
\begin{array}{ccccc}
 1 & 0 & 0 & 0 & 0 \\
 0 & 1 & 0 & 0 & 0 \\
 0 & 0 & 0 & 0 & 1 \\
 0 & 0 & -1 & 1 & 1 \\
 0 & 0 & 1 & 0 & 0 \\
\end{array}
\right),\ 
VR_4V^{-1}=\left(
\begin{array}{ccccc}
 1 & 0 & 0 & 0 & 0 \\
 0 & 0 & 1 & 0 & 0 \\
 0 & 1 & 0 & 0 & 0 \\
 0 & 0 & 0 & 1 & 0 \\
 0 & -1 & 1 & 0 & 1 \\
\end{array}
\right),$$
and
$$
VR_5V^{-1}= \left(
\begin{array}{ccccc}
 0 & 1 & 0 & 0 & 0 \\
 1 & 0 & 0 & 0 & 0 \\
 0 & 0 & 1 & 0 & 0 \\
 0 & 0 & 0 & 1 & 0 \\
 0 & 0 & 0 & 0 & 1 \\
\end{array}
\right).
$$
These act on the {\it left} on the set of bends of the packing; the bends which arise in the packing are those in the action of the above matrices on the initial bends vector $(1,1,1,1,1)$ of all bends equal to $1$.
All but the first matrix are themselves integral, and hence evidently preserve integrality. Moreover, inspection shows that they preserve the congruence $b_1+b_2+2b_3+3b_4+b_5\equiv0\pmod 4$, which is satisfied by the initial bends vector $(1,1,1,1,1)$.
This congruence ensures that the action of the first matrix $VR_1 V^{-1}$ always gives integer bends in the orbit, thus proving the proper integrality of the packing.

Finally, following the proof of the Arithmeticity Theorem, we exhibit
a Lorentzian lattice $L$ so that
the arithmetic lattice
$O_L(\Z)$
contains a conjugate of $\Gamma$. We find in this case that $L$ is {\it not} reflective.
To begin, we record that the Gram matrix of $V$:
$$
G = G(V) = VQV^T=-\frac{1}{7}
\left(
\begin{array}{ccccc}
 -7 & 9 & 9 & 9 & 9 \\
 9 & -7 & 9 & 9 & 9 \\
 9 & 9 & -7 & 9 & 25 \\
 9 & 9 & 9 & -7 & 41 \\
 9 & 9 & 25 & 41 & -7 \\
\end{array}
\right),
$$
where $Q = \mathbb H \perp \langle 1,1,1\rangle$.
The inverse Gram matrix equals
$$
G^{-1}=
-\frac{7}{256}
\left(
\begin{array}{ccccc}
 25 & 9 & 0 & -9 & -9 \\
 9 & 25 & 0 & -9 & -9 \\
 0 & 0 & 12 & -8 & -4 \\
 -9 & -9 & -8 & 9 & 1 \\
 -9 & -9 & -4 & 1 & 5 \\
\end{array}
\right),
$$
and gives rise to the lattice $L$ having Lorentzian quadratic form $f(x) = x^T G^{-1} x$.
The proof of the Arithmeticity Theorem \cite{KN} shows that $V\Gamma V^{-1} <  \OO_L(\Z).$
Finally, Vinberg's algorithm applied to $L$ produces the following first $31$ roots:

\begin{align*}
    &e_1 = (1, -1, 0, 0, 0),\ e_2 = (0, 1, 2, 1, 0),\  e_3 = (0, 1, 0, 1, 0), \\
    &e_4 = (-1, 0, -1, -2, 1),\  e_5 = (0, 1, 2, 3, 0),\ e_6 = (2, 4, 9, 15, -1), \\
    &e_7 = (1, 2, 8, 11, 0),\ e_8 = (1, 3, 6, 12, 0),\ e_9 = (5, 6, 20, 29, -2), \\
    &e_{10} = (8, 8, 21, 35, -3),\ e_{11} = (7, 10, 23, 36, -3),\ e_{12} = (4, 4, 15, 25, -1), \\
    &e_{13} = (6, 7, 16, 31, -2),\ e_{14} = (3, 3, 7, 13, -1),\ e_{15} = (5, 6, 13, 28, -1), \\
    &e_{16} = (5, 6, 18, 29, -2),\ e_{17} = (8, 8, 23, 35, -3),\ e_{18} = (2, 3, 11, 20, 1), \\
    &e_{19} = (4, 4, 17, 25, -1),\ e_{20} = (1, 1, 2, 5, 0),\ e_{21} = (3, 6, 19, 26, -1), \\
    &e_{22} = (3, 6, 15, 26, -1),\ e_{23} = (4, 4, 9, 23, 1),\ e_{24} = (3, 3, 14, 22, 0), \\
    &e_{25} = (3, 3, 11, 17, -1),\ e_{26} = (3, 4, 12, 23, 0),\ e_{27} = (2, 5, 8, 21, 2), \\
    &e_{28} = (3, 3, 7, 12, -1),\ e_{29} = (11, 12, 34, 59, -4),\ e_{30} = (10, 11, 32, 49, -4), \\
    &e_{31} = (8, 11, 35, 54, -3),\; \ldots
\end{align*}

There is an infinite order symmetry $\sigma$ mapping the elliptic subdiagram spanned by the roots $\{ e_{13}, e_3, e_{14}, e_1, e_6 \}$ to another elliptic subdiagram spanned by the roots $\{ e_{12}, e_{19}, e_{25}, e_{31}, e_7 \}$, exactly in this order 
\begin{equation*}
    e_{13} \mapsto e_{12}, e_3 \mapsto e_{19}, e_{14} \mapsto e_{25}, e_1 \mapsto e_{31}, e_6 \mapsto e_7.
\end{equation*}

The matrix from $\OO_L(\mathbb{Z})$ representing $\sigma$ in the respective bases is 
\begin{equation*}
  S =  \left(
\begin{array}{ccccc}
54 & 41 & -3 & -10 & 79 \\
64 & 48 & -3 & -12 & 96 \\
159 & 122 & -8 & -30 & 238 \\
276 & 210 & -15 & -51 & 413 \\
-21 & -16 & 1 & 4 & -31
\end{array}
\right).
\end{equation*}
Thus the fundamental polyhedron of the maximal reflection subgroup of the  arithmetic group $\OO_L(\mathbb{Z})$ containing $\Gamma$ has an infinite order symmetry; hence $L$ is \textit{not} reflective.

However, we would like to remark that the group $\OO_L(\mathbb{Z})$ is defined up to commensurability (because of the choice of normals described above), and reflectivity is \textit{not} commensurability invariant in the setting of Lorentzian lattices. Indeed, the lattice
\begin{equation*}
    \left(
\begin{array}{ccccc}
0 & 0 & 49 \\ 
0 & 49 & 7 \\ 
49 & 7 & 3
\end{array}
\right)
\end{equation*}
has \textit{no} roots \cite{BK22, Col}, while it is easily seen to be commensurable to the diagonal lattice
\begin{equation*}
    \left(
\begin{array}{ccccc}
-1 & 0 & 0 \\ 
0  & 2 & 0 \\ 
0  & 0 & 1
\end{array}
\right),
\end{equation*}
that is reflective and corresponds to the $(2,4,\infty)$ hyperbolic triangle reflection group.

\end{document}